\documentclass[12pt, a4paper, pdftex]{amsart} 

\title{The circle transfer and cobordism categories} 

\author{Jeffrey Giansiracusa} 
\address{Department of Mathematics\\Swansea University\\ Singleton Park\\Swansea SA2 8PP\\UK} 
\email{j.h.giansiracusa@swansea.ac.uk}
\date{10 May, 2018}

\usepackage{mathptmx}
\DeclareSymbolFont{cmlargesymbols}{OMX}{cmex}{m}{n}
\DeclareMathSymbol{\mycoprod}{\mathop}{cmlargesymbols}{"60}
\let\coprod\mycoprod

\usepackage[protrusion=true,expansion=true]{microtype}

\usepackage{amssymb}
\usepackage{mathrsfs}  
\usepackage{latexsym}
\usepackage{amsmath}

\usepackage{rotating} 

\usepackage[cmtip,arrow]{xy}
\usepackage{pb-diagram,pb-xy}

\usepackage[hmargin=3cm, vmargin=3cm]{geometry}
\newtheorem{theorem}{Theorem}[section] 
\newtheorem{lemma}[theorem]{Lemma} 
\newtheorem{proposition}[theorem]{Proposition}

\newcommand{\Obj}{\mathrm{Obj}}
\newcommand{\Mor}{\mathrm{Mor}}
\newcommand{\Map}{\mathrm{Map}}
\newcommand{\Diff}{\mathrm{Diff}}
\newcommand{\Aut}{\mathrm{Aut}}
\newcommand{\Emb}{\mathrm{Emb}}

\newcommand{\id}{\mathrm{id}}

\newcommand{\C}{\mathbb{C}}
\newcommand{\R}{\mathbb{R}}

\newcommand{\Smash}{\wedge}
\newcommand{\trf}{\mathrm{trf}}
\newcommand{\ev}{\mathrm{ev}}

\DeclareMathOperator*{\colim}{colim} 

\newcommand{\CIRC}{{\mathcal{C}\mathit{irc}}}
\newcommand{\COB}{{\mathscr{C}_1}}
\newcommand{\SURF}{{\mathscr{C}_2}}
\newcommand{\CYL}{{\mathcal{C}\mathit{yl}}}

\begin{document}

\begin{abstract}
  The circle transfer $Q\Sigma (LX_{hS^1})_+ \to QLX_+$ has appeared in several contexts in
  topology.  In this note we observe that this map admits a geometric re-interpretation as a
  morphism of cobordism categories of 0-manifolds and 1-cobordisms.  Let $\COB(X)$ denote the
  1-dimensional cobordism category and let $\CIRC(X)\subset \COB(X)$ denote the subcategory whose
  objects are disjoint unions of unparametrised circles.  Multiplication in $S^1$
  induces a functor $\CIRC(X) \to \CIRC(LX)$, and the composition of this functor
  with the inclusion of $\CIRC(LX)$ into $\COB(LX)$ is homotopic to the circle transfer.  As a
  corollary, we describe the inclusion of the subcategory of cylinders into the 2-dimensional
  cobordism category $\SURF(X)$ and find that it is null-homotopic when $X$ is a point.
\end{abstract}
\maketitle

\section{Introduction}\label{sec-introduction}

The circle transfer from stable homotopy theory appears in many contexts in algebraic topology.
The purpose of this note is to relate the circle transfer to certain low-dimensional cobordism
categories.  One the one hand, this provides a somewhat novel geometric description of the circle
transfer, and on the other hand it provides a description of the inclusion of the subcategory of
cylinders into the full surface cobordism category.

Let $Y$ be a based space with an action of the circle group $S^1$.  The circle transfer in its
general form is an infinite loop map
\[
\mathrm{trf}: Q \Sigma Y_{hS^1} \to QY,
\]
where $Y_{hS^1} = ES^1 \times_{S^1} Y$ is the homotopy quotient.  An important special case is when
$Y=LX_+$ is the free loop space on some space $X$ plus a disjoint basepoint; the resulting circle
transfer in this case is then
\begin{equation}\label{trf-X}
Q \Sigma (LX_{hS^1})_+ \to QLX_+,
\end{equation}
and this map is central to the construction of topological cyclic homology and the trace map that
relates it with Waldhausen's $A(X)$ and the smooth Whitehead space; see e.g., \cite{BHM} and
\cite{Rognes}.  Taking $Y$ to a be a single point, the circle transfer is a map
$Q\Sigma BS^1_+ \to QS^0$ that has been studied by stable homotopy theorists; in particular
its image in the stable homotopy groups of spheres has been examined by many authors, such as
\cite{Mukai1, Mukai2}, \cite{Imaoka}, \cite{Miller}, and \cite{baker1, baker2}.

\subsection{Statement of the results}

The first result of this note is a cobordism category model for the circle transfer.  Let $\COB(X)$
denote the topological category whose objects are compact oriented 0-manifolds with maps to $X$ and whose
morphism are moduli spaces of 1-dimensional cobordisms with maps to $X$.  Let
$\CIRC(X) \subset \COB(X)$ denote the full subcategory on the single object consisting of the empty
0-manifold.  We will define these categories in detail in section \ref{sec:cobordism-categories}.  The
multiplication map $\mu: S^1 \times S^1 \to S^1$ induces a functor $\mu^*: \CIRC(X) \to \CIRC(LX)$.

\begin{theorem}\label{big-theorem-X}
There is a homotopy commutative diagram,
\[
\begin{diagram}
\node{Q\Sigma (LX_{hS^1})_+} \arrow[2]{e,t}{\mathrm{trf}} \arrow{s,l}{\mbox{\begin{sideways}$\simeq$\end{sideways}}}
\node[2]{QLX_+} \\
\node{|\CIRC(X)|} \arrow{e,t}{\mu^*} \node{|\CIRC(LX)|} \arrow{e,J} \node{|\COB(LX)|.}
\arrow{n,lr}{\mbox{\begin{sideways}$\simeq$\end{sideways}}}{\mathrm{PT}}
\end{diagram}
\] 
\end{theorem}
The vertical arrow on the left is given essentially by including a subspace of standard round circles into the space
of all circles. The vertical arrow on the right is a Pontrjagin-Thom collapse
map, and it is a weak equivalence by the 1-dimensional case of the main theorem from \cite{GMTW}.

The second result concerns cobordism categories of surfaces.  Let $\SURF(X)$ denote the category
whose objects are closed oriented 1-manifolds in $\mathbb{R}^\infty$ and whose morphisms are
2-dimensional oriented cobordisms; we will define this in detail below in section
\ref{sec:cobordism-categories}. 

At the level of spectra the circle transfer (for $X=*$) fits into a cofibre sequence
\[
\Sigma^\infty BS^1_+ \stackrel{\trf}{\to}  \Sigma^\infty S^0 \to \Sigma^2 MTSO(2),
\]
where $MTSO(2)$ denotes the Madsen-Tillmann spectrum for $BSO(2)$, i.e., Thom spectrum of the
negative of the tautological 2-plane bundle over $BSO(2)$; see \cite{Madsen-Schlichtkrull}.  There
is a Pontrjagin-Thom map
\[
\mathrm{PT}: |\SURF(X)| \to \Omega^{\infty-1} MTSO(2)\wedge X_+
\] 
that is a weak equivalence by \cite{GMTW}.  Let $\CYL(X)\subset \SURF(X)$ denote the subcategory
with the same objects and only those morphisms consisting of cobordisms that are unions of
cylinders, each with one incoming end and one outgoing end.  Note that the $E_\infty$-structure on
$|\SURF(X)|$ restricts to an $E_\infty$ structure on $|\CYL(X)|$, but this subcategory is not group
complete and so we consider its group completion $\Omega B |\CYL(X)|$.

\begin{theorem}\label{thm2}
There is a homotopy equivalence $Q(LX_{hS^1})_+ \stackrel{\simeq}{\to} \Omega B |\CYL(X)|$, and there is
a homotopy commutative diagram
\[
\begin{diagram}
  \node{\Omega B |\CYL(X)|} \arrow[3]{e} \node[3]{|\SURF(X)|}
  \arrow{s,l}{\mbox{\begin{sideways}$\simeq$\end{sideways}}} \\
  \node{Q(LX_{hS^1})_+}
  \arrow{n,l}{\mbox{\begin{sideways}$\simeq$\end{sideways}}}
  \arrow{e,t}{\mathrm{trf}} \node{\Omega QLX_+} \arrow{e,t}{\mathrm{ev}}
  \node{\Omega QX_+} \arrow{e} \node{\Omega^{\infty-1}(MTSO(2)\Smash
    X_+),}
\end{diagram}
\]
where the lower right horizontal arrow is induced by smashing the cofibre of the circle transfer for
a point with $X_+$ and then applying $\Omega^\infty$.  In particular, when $X$ is a point then the
inclusion $|\CYL(*)| \hookrightarrow |\SURF(*)|$ is null-homotopic.
\end{theorem}

\subsection{Conventions} Throughout this paper all manifolds are assumed smooth compact and
oriented, and all diffeomorphisms are orientation-preserving.  The geometric realisation of the nerve of a category $\mathscr{C}$ will be denoted
$|\mathscr{C}|$, and we reserve $B$ for deloopings and classifying spaces of monoids..

\section{Construction of the circle transfer}

In this section we briefly recall a geometric construction of the circle transfer.  A more detailed
and general account of such transfer maps is contained in \cite{Madsen-Schlichtkrull} and
\cite{May-equivariant}.  The origin of the circle transfer lies in $S^1$-equivariant homotopy
theory.  Let $Y$ be a based $S^1$-CW complex. There is an equivariant transfer map
\[
 \mathrm{trf}^{S^1}: Q\Sigma Y_{hS^1}
 \stackrel{\simeq}{\longrightarrow} (Q_{S^1} Y)^{S^1}
\]
that is a weak equivalence, and composing this with the map $(Q_{S^1}Y)^{S^1} \to Q Y$ that
includes the fixed points and forgets the circle action yields the circle transfer that we shall be
concerned with.

We will instead use the concrete geometric description that avoids equivariant infinite loop spaces,
following \cite{Madsen-Schlichtkrull}[section 2].  We will restrict attention to the special case
when $Y$ is of the form $Z_+$ for some space $Z$.  First we fix a model for $ES^1$, and for this we
will use the unit sphere $S(\C^\infty)$ inside $\C^\infty = \colim_n \C^n$ with the standard action
of $S^1 \subset \C^*$.  There is an embedding
\begin{equation}\label{es1emb}
j: S(\C^n) \times Z \hookrightarrow \C^n \times \R \times ( S(\C^n) \times_{S^1} Z)
\end{equation}
given by the inclusion $S(\C^n) \hookrightarrow \C^n$, the zero map on $\R$, and the quotient map to
$S(\C^n) \times Z \to S(\C^n)\times_{S^1} Z$.  This embedding is equivariant, where $S^1$ acts as
usual on $S(\C^n)$ and trivially on the other two factors.  Moreover, it has a normal bundle $N(j)$
that can be equivariantly embedded as a tubular neighbourhood.

\begin{lemma}
  The normal bundle $N(j)$ of $S(\C^n) \times Z$ in $\C^n \times \R \times (S(\C^n) \times_{S^1} Z )$
  is equivariantly isomorphic to $S(\C^n) \times Z \times \C^n$.
\end{lemma}
\begin{proof}
  We consider $\C$ as a real vector space with $S^1$ action.  We give an explicit formula for an
  equivariant embedding of $S(\C^n) \times Z \times \C^n$ into
  $\C^n \times \R \times (S(\C^n) \times_{S^1} Z )$ as a tubular neighbourhood of $j$.  Given a unit
  vector $u \in S(\C^n)$, let $\tau_u$ denote the positive direction unit tangent vector to the
  orbit of $u$ at $u$. The desired formula is
\[
(u, z, v) \mapsto \left( 
u + \frac{v}{1+|v|},  \frac{\langle v, \tau_u\rangle }{(1+|v|)}, [u,z] \right).
\]
Clearly this formula agrees with $j$ on the zero section.  When $v$ is nonzero, we see that this is
an embedding as follows.  Given a triple
$(w,t,\overline{z}) \in \C^n \times \R \times (S(\C^n) \times_{S^1} Z )$, we will reconstruct the unique triple
$(u,z,v)$ mapping to it. First note that $w$ is at most a distance $1/4$ from $u$.  The component
$\overline{z} \in S(\C^n) \times_{S^1} Z $ determines an orbit $O$ in $S(\C^n)$, and there will be exactly two points $u \in O$ for which
$\langle w - u, \tau_u \rangle = t$; moreover, only one of these two points will be sufficiently close to $w$
(with the other being antipodal).  Thus $u$ is uniquely determined, and hence $v/(1+|v|) = w-u$ is
also uniquely determined.  Finally, $u$ and $[u,z]$ together uniquely determine $z \in Z$.
\end{proof}

Applying Thom collapse to the above tubular neighbourhood, one obtains a map
\[
\Sigma^{2n+1} (S(\C^n) \times_{S^1} Z)_+ \to \Sigma^{2n} (S(\C^n)\times Z)_+,
\]
whose adjoint is a map
$\Sigma (S(\C^n) \times_{S^1} Z)_+ \to \Omega^{2n}\Sigma^{2n} (S(\C^n)\times Z)_+.$
Taking the colimit as $n\to\infty$ and using the fact that the sphere $S(\C^\infty)$  is
contractible, we obtain a map of based spaces
\[
\Sigma (Z_{hS^1})_+ \to Q Z_+
\]
that is natural with respect to $S^1$-equivariant maps $Z \to Z'$.
This map is sometimes called the circle transfer; however, we will call
it the \emph{pre-circle transfer} and reserve the term \emph{circle
transfer} for the natural $\Omega^\infty$ extension of the pre-circle
transfer to $Q\Sigma(Z_{hS^1})_+$.

In order to relate the circle transfer to cobordism categories, we will need the following result.

\begin{lemma}\label{normal-restriction}
  Suppose $\pi: Y \to B$ is a fibre bundle with smooth fibres, $V$ is
  a smooth manifold, and $i: Y \to V$ is a map that restricts to an embedding on each fibre. One
  may form the embedding $i\times \pi: Y \hookrightarrow V\times B$,
  with normal bundle denoted $N(i\times \pi)$.  Then for any fibre
  $F$ of $\pi$ there is a natural isomorphism
  \[N(i|_F) \cong N(i\times \pi) |_F.\]
\end{lemma}
\begin{proof}
  Since the statement is local to a single fibre and the fibre bundle
  is locally trivial it suffices to consider the case of a trivial
  bundle $Y = F \times B$.  In this case one immediately sees that the
  normal bundle of $i \times \pi :F\times B \hookrightarrow V\times
  B$, when restricted to $F\times \{b\}$ is precisely the normal
  bundle of $i|_{F\times\{b\}}: F\times \{b\} \hookrightarrow V$.
\end{proof}

When the fibre bundle is $\pi: S(\C^n) \times Z \to S(\C^n) \times_{S^1} Z$ and the map $i$ is the composition 
\[
i: S(\C^n) \times Z \stackrel{\mathrm{proj}}{\longrightarrow} S(\C^n)
\stackrel{\mathrm{id}\times 0}{\hookrightarrow} \C^n\times \R,
\] 
we may use the identification of Lemma \ref{normal-restriction} to
construct the pre-circle transfer one orbit at a time.

\begin{lemma}\label{orbit-trf}
  The pre-circle transfer sends an orbit $O \subset ES^1 \times Z$ to the point
  of $\Omega Q Z_+$ obtained by Pontrjagin--Thom collapse applied to
  the diagram
\[
\C^n \times Z \leftarrow N(i\times \pi)|_{O} \cong
N(i|_{O}) \hookrightarrow \C^n \times \R.
\]
where the right arrow is an open inclusion as a tubular neighbourhood, and the left arrow is a
vector bundle morphisms that is an isomorphism on each fibre (it covers the projection $S(\C^n) \times Z \to Z$).
\end{lemma}

\section{Cobordism categories}
\label{sec:cobordism-categories}

We now define the cobordism categories $\mathscr{C}_d(X)$ of $(d-1)$-manifolds and $d$-dimensional
cobordisms, essentially following the construction of \cite{Madsen-Tillmann}[\S 2.1] or
\cite{GMTW}[\S 2.1];  $\mathscr{C}_d$ will be a topological category, meaning that it consists of a
space of objects, a space of morphisms, and continuous maps defining source, target and compositions
of arrows.

As a set, the objects $\Obj_{\mathscr{C}_d}$ will be the set of quadruples  $(M, f, \phi, a)$ where 
\begin{itemize}
\item $M$ is a smooth closed oriented $(d-1)$-dimensional submanifold of $\R^n$ for some $n$; 
\item $f: M \to X$ is a continuous map;
\item $\phi$ is an embedding of the normal bundle of $M$ as a tubular neighbourhood; 
\item $a\in \R$. 
\end{itemize}
The above data is considered up to the equivalence generated by the stabilisations $\R^n
\hookrightarrow \R^{n+1}$.

In order to define the topology, first fix a manifold $M$ and a rank $(n-d+1)$ vector bundle
$\mathcal{N}$ over it.  There is a map
$\Emb(\mathcal{N},\R^n) \to \Emb(\mathcal{N}\oplus \R, \R^{n+1})$ given by choosing a bounded
embedding of $\R \to \R$, such as $f(x) = x/(1+|x|)$ and using this to extend each embedding of
$\mathcal{N}$ to $\mathcal{N}\oplus \R$.  Let $\Aut(\mathcal{N})$ denote the group of bundle
automorphism of $\mathcal{N}$ that cover a diffeomorphism of $M$.  We give $\Emb(\mathcal{N},\R^n)$
and $\Aut(\mathcal{N}$ the $\C^\infty$ Frechet topology.  The space of objects is then
\[
\left( \colim_{n\to \infty} \coprod_{(M^{d-1},\mathcal{N}^{n-d+1})} \Emb(\mathcal{N},\R^n) /
  \Aut(\mathcal{N}) \right) \times \R,
\]
where the disjoint union runs over all isomorphism classes of pairs.

The set of morphisms $\Mor_{\mathscr{C}_d}$ consists the identity morphisms disjoint union with the set of
all tuples $(W,f,\phi,a_1,a_2)$ where
\begin{itemize}
\item $W$ is a $d$-dimensional submanifold, with boundary, embedded in $[a_1,a_2] \times \R^\infty$ such that
  on a neighbourhood of each endpoint $a_i$ of the interval it is the product of a
  $(d-1)$-submanifold $\partial_i W \subset \R^\infty$
  with an interval; i.e., there exists $\epsilon
  > 0$ such that $W\cap \left( [a_1, a_1+\epsilon) \times \R^\infty \right)= [a_1, a_1+\epsilon)
  \times \partial_1 W$, and likewise near $a_2$;
\item $f: W \to X$ is a continuous map; 
\item $\phi$ is an embedding of the normal bundle of $W$ as a tubular neighbourhood such that the
  fibres over the boundary are embedded over $\{a_1\}$ and $\{a_2\}$, and such that $\phi$ is
  independent of $t\in [a_1,a_2]$ for $t$ sufficiently close to the endpoints;
\item $a_1$ and $a_2$ are real numbers satisfying $a_1 < a_2$
\end{itemize}
For fixed $a_1$ and $a_2$, as for the objects, we topologise this as we did for the objects, as a
colimit of normal bundle embeddings modulo bundle automorphisms.  Linear scaling identifies the
resulting space for any pair $a_1 < a_2$ with the space for any other pair $a_1' < a_2'$, and thus
we topologise the full space of non-identity morphisms as a bundle over the open half-plane of all
pairs $a_1 < a_2$.  The source and target maps are given by restricting to either end of the
interval.

\section{The Pontrjagin-Thom map}

Now we briefly recall the Pontrjagin-Thom construction from \cite{Madsen-Tillmann}[\S 2.1] yields,
up to homotopy, a map
\[
\mathrm{PT}: |\mathscr{C}_d(X)| \to \Omega^{\infty-1}MTSO(d)\wedge X_+ 
\]
that is a weak homotopy equivalence by the main theorem of \cite{GMTW}.

First we require an appropriate model for the target, and for this we use the Moore path category.  For
a space $Y$, let $\mathcal{P}(Y)$ denote the topological category with:
\begin{itemize}
\item $\Obj_{\mathcal{P}(Y)} = Y$,
\item $\Mor_{\mathcal{P}(Y)} = \bigcup_{t\in [0,\infty)}
  \Map([0,t],Y)$, topologised compatibly with the topology of
  $[0,\infty)$.
\end{itemize}
The source and target maps are given by restricting to either
endpoint, and composition of morphisms is concatenation of paths.
Identity morphisms are the paths of length zero.

\begin{proposition}\label{path-prop}
There is a natural weak equivalence $Y \stackrel{\simeq}{\to} |\mathcal{P}(Y)|$.
\end{proposition}
\begin{proof}
  Considering $Y$ as a constant simplicial space, the simplicial map
  $Y_\bullet \to N_\bullet \mathcal{P}(Y)$, which in degree $n$ sends
  $y$ to the constant path at $y$ of length zero, is clearly a homotopy
  equivalence degree-wise.  Hence it induces a homotopy equivalence of
  geometric realisations.
\end{proof}

We will now produce a map $|\mathscr{C}_d(X)| \to |\mathcal{P}(\Omega^{\infty-1}MTSO(d)\wedge X_+)|$
by defining a continuous functor between these categories.  First consider an object
$Z = (M,f,\phi,a) \in \Obj_{\mathscr{C}_d(X)}$.  It is represented by a $(d-1)$ manifold $M$ embedded
in $\R^n\times \{a\}$ with a tubular neighbourhood $N_M \hookrightarrow \R^n\times \{a\}$ and a map $f$
to $X$; we think of this as the germ of a $d$-manifold with tubular neighbourhood
$M\times (a-\epsilon, a+\epsilon)$ in $\R^n \times (a-\epsilon, a+\epsilon)$.  Let $Gr(d,n)$ denote
the Grassmannian of oriented $d$-planes in $\R^n$, and let $\gamma_{d,n}$ denote the tautological
$d$-plane bundle over it.  The tangent bundle of $M$ is classified by a map $M \to Gr(d,n+1)$. that
factors through the stabilization map $Gr(d-1,n) \to Gr(d,n+1)$.  Taking Thom spaces yields a map
\[
\mathrm{Th}(N_M) \to \mathrm{Th}(\gamma_{d,n+1}^\perp)\wedge X_+,
\]
where the map to the factor $X$ is given by the bundle projection $N_M \to M$ followed by $f$.  Collapsing the
complement of the tubular neighbourhood yields a map $S^n \to \mathrm{Th}(N_M)$ and so the
composition is a point in $\Omega^n \mathrm{Th}(\gamma_{d,n+1}^\perp) \wedge X_+$, and hence letting $n$ go to
infinity one obtains a continuous map
\[
\Obj_{\mathscr{C}_d} \to \Omega^{\infty-1} MTSO(d) \wedge X_+.
\]

Now consider a morphism; it is represented by a $d$-manifold $W$ embedded
in $\R^n \times [a_1, a_2]$ with boundary lying over the endpoints, and equipped with a tubular
neighbourhood and a map to $X$.  The Pontrjagin-Thom collapse produces a path of length $a_2 - a_1$ in the space
$\Omega^n \mathrm{Th}(\gamma_{d,n+1}^\perp)\wedge X_+$, and hence one obtains a continuous map
\[
\Mor_{\mathscr{C}_d} \to \Mor_{\mathcal{P}(\Omega^{\infty-1}MTSO(d) \wedge X_+)}
\] 
that is compatible with source, target, and compositions.

Note that in the special case of $d=1$, the space $BSO(1)$ is contractible, and so there is an
equivalence $\Omega^{\infty -1} MTSO(1) \wedge X_+ \simeq QX_+$.

\section{The subcategory of circles}

We now consider the full subcategory $\CIRC(X) \subset \COB(X)$ consisting of cobordisms from the
empty 0-manifold to itself.  Morphisms in this subcategory are disjoint unions of circles with
tubular neighbourhoods in $[a_1,a_2]\times \R^\infty$.  It is thus straightforward to see that the
space of morphism in $\CIRC(X)$ has the homotopy type
\[
\coprod_k E\Sigma_k \times_{\Sigma_k} (LX_{hS^1}).
\]
since the diffeomorphism group of a union of $k$ circles is homotopy equivalent to the wreath
product of $\Sigma_k$ with $(S^1)^k$.  Moreover, composition of morphisms corresponds to disjoint
union.  By thinking of a point of the homotopy quotient $LX_{hS^1}$ as an unparametrized circle in
$\R^\infty$ with a map to $X$, one sees that there is a map
\[
LX_{hS^1} \to \Mor_{\CIRC(X)}.
\]

\begin{proposition}
The above map induces a weak equivalence $Q\Sigma (LX_{hS^1})_+ \to |\CIRC(X)|.$
\end{proposition}
\begin{proof}
This follows from the Barrat-Priddy-Quillen-Segal Theorem.
\end{proof}

The multiplication map $\mu: S^1 \times S^1 \to S^1$ induces a continuous functor $\mu^*: \CIRC(X) \to
\CIRC(LX)$.  Note that this functor is does not extend to $\COB(LX)$.

\section{The circle transfer as a collapse map}

\begin{proposition}\label{prop:trf-ev}
Given any space $X$, there is a homotopy commutative diagram 
\[
\begin{diagram}
\node{|\CIRC(X)|} \arrow{e,J} \node{|\COB(X)|} \arrow{s,lr}{\mbox{\begin{sideways}$\simeq$\end{sideways}}}{\mathrm{PT}} \\
\node{Q\Sigma (LX_{hS^1})_+}\arrow{n,l}{\mbox{\begin{sideways}$\simeq$\end{sideways}}} \arrow{e,t}{\trf \:\circ \:\ev} \node{QX_+.}
\end{diagram}
\]
\end{proposition}

\begin{proof}
  First note that a point of $LX_{hS^1} = S(\C^\infty) \times_{S^1} LX$ consist of a pair
  $(O,\alpha)$ where $O\subset S(\C^\infty)$ is a circle orbit and $\alpha: O \to X$ is a continuous
  map; the map  $S(\C^\infty)_\times LX \to S(\C^\infty) \times_{S^1} LX$ is a principal $S^1$-bundle and a
  point of the total space $S(\C^\infty) \times LX$ is represented by a triple $(O,t,\alpha)$ where
  $t$ is a point on $O$.

  Both directions around the diagram are the adjoints of maps
  $\Sigma^{\infty+1} (S(\C^\infty) \times_{S^1} LX)_+ \to \Sigma^\infty X_+$, which are in
  represented by families of maps
  \[\Sigma^{2n+1} (S(\C^n) \times_{S^1} LX)_+ \to \Sigma^{2n} X_+,\] so we need only check that these are
  homotopic. Consider a point $p$ of $\Sigma^{2n+1} (S(\C^n) \times_{S^1} LX)_+$; it corresponds to a
  triple $(v,O,\alpha)$, where $v \in \C^n \times \R$.  Using Lemma \ref{orbit-trf}, we see that the
  pre-transfer sends this to the basepoint if $v$ is outside the tubular neighbourhood of $O$, and
  if it is in the tubular neighbourhood we get a point of $\C^n$ (since the normal bundle is
  trivialised with fibre $\C^n$) and a point $t\in O$ recording which fibre $v$ is in. Composing
  with the loop evaluation map then replaces $\alpha \in \Map(O,X)\cong LX$ with $\alpha(t) \in X$.

On the other hand, the anti-clockwise composition takes a point $(v,O,\alpha)$ and regards is as a
submanifold in $\C^n \times \R$ equipped with a map to $X$ plus a point $v$ in the ambient space,
and does Thom collapse.  One immediately sees that this is thus the same as the pre-transfer
composed with loop evaluation.
\end{proof}

We now deduce our first theorem from the above proposition.


\begin{proof}[Proof of Theorem \ref{big-theorem-X}]
  Let $\ev: L^2X \to LX$ denote the loop evaluation map for the outer $L$, as well as the maps
  induced by this. We consider $L^2X$ with the circle action given by rotating the outer $L$.
  Now consider the diagram below:
\[
\begin{diagram}
\node[2]{QLX_+} \arrow{se,b}{\mu^*} \arrow{see,t}{\id}\\
\node{Q\Sigma (LX_{hS^1})_+} \arrow{s,r}{} \arrow{e,t}{\mu^*} \arrow{ne,t}{\trf}
\node{Q\Sigma (L^2 X_{hS^1})_+} \arrow{s,r}{}  \arrow{e,t}{\trf}
\node{Q L^2X_+} \arrow{e,t}{\ev}
\node{QLX_+}
\\
\node{|\CIRC(X)|} \arrow{e,t}{\mu^*} 
\node{|\CIRC(LX)|} \arrow[2]{e,J}
\node[2]{|\COB(LX)|.} \arrow{n,r}{\mathrm{PT}}
\end{diagram}
\]
The lower left square commutes since the horizontal arrows are both induced by the multiplication
map $\mu$.  The lower right square commutes up to homotopy by Proposition \ref{prop:trf-ev} applied to
the space $LX$.  The upper left part commutes since the circle transfer is natural with respect to
$S^1$-equivariant maps.  The upper right part commutes since the composition
\[
\{1\} \times S^1 \hookrightarrow S^1 \times S^1 \stackrel{\mu}{\to} S^1
\]
is equal to the identity on $S^1$.
\end{proof}


\section{The cylinder subcategory of the surface category}

Consider the 2-dimensional cobordism category $\SURF(X)$.  In this section we study the
subcategory $\CYL(X)$ whose morphisms are only those surfaces that are topologically trivial cobordisms
(i.e., a disjoint union of cylinders) from a collection of circles to another collection of the same
number of circles.

\begin{lemma}
There is a homotopy equivalence 
$\Obj_{\CYL(X)} \simeq \coprod_{k \in \mathbb{N}} (LX_{hS^1})^k \times_{\Sigma_k} E\Sigma_k.$
\end{lemma}
\begin{proof}
Any object is a disjoint union of circles, so the space of objects breaks up as a disjoint union
with components corresponding to the number of circles.  The component corresponding to $k$ circles
has the homotopy type
\[
\left. \Emb\left(\coprod_k S^1, \mathbb{R}^\infty \right) \times LX \middle/ \Diff\left(\coprod_k
    S^1\right) \right. \simeq (LX_{hS^1})^k \times_{\Sigma_k} E\Sigma_k,
\]
since the space of embeddings is contractible by the Whitney embedding theorem and the
diffeomorphism group acts freely and locally trivially.
\end{proof}

\begin{lemma}
The inclusion $\Obj_{\CYL(X)} \hookrightarrow |\CYL(X)|$ is a weak homotopy equivalence.
\end{lemma}
\begin{proof}
Let $M$ denote the additive topological monoid $\{0\}\sqcup (0,\infty)$.  There is a homomorphism from
$M$ to the multiplicative monoid $M' =\{0,1\}$ given by sending $0 \mapsto 1$ and all other elements
to $0$, and this homomorphism induces a level-wise homotopy equivalence on nerves, and hence a weak
homotopy equivalence of geometric realizations.  The realization of $M'$ is contractible since it is a $K(\pi,1)$
for which the group $\pi$ is the group completion of $M'$, and this group completion is trivial.

Now consider the simplicial space
\[
Z_\bullet = \Obj_{\CYL(X)} \times N_\bullet M.
 \]
 Since the geometric realization of $N_\bullet M$ is weakly contractible, it follows that the
 inclusion $\Obj_{\CYL(X)} \hookrightarrow |Z_\bullet|$ is a weak equivalence.

We will now define a map of simplicial spaces
\[
\gamma: Z_\bullet \to N_\bullet \CYL(X).
\]
A diffeomorphism $\varphi$ of a circle extends to a diffeomorphism $\varphi\times \mathrm{id}_{[0,a]}$ of
the cylinder $S^1 \times [0,a]$, and a map $S^1 \to X$ extends to a map $S^1 \times [0,a] \to X$
constant in the interval direction.  Extending in this way induces a map $e_a$ from $\Obj_{\CYL(X)}$
to the space $\Mor^0_{\CYL(X)}$ of non-identity morphisms.  The map $\gamma$ is the identity on
0-simplices; on 1-simples it is given by
\[
\left(a \in M, X\in \Obj_{\CYL(X)}\right) \mapsto \begin{cases} e_a(X) & a\neq 0 \\ \mathrm{Id}_X & a=0, \end{cases}
\]
and this also determines in on higher simplices.

Finally, we show that $\gamma$ induces a weak equivalence on geometric realizations.  The extension
map $e$ is a section of the source map, which is induced by restricting to the left end of each
cylinder. Since the restriction map $\Diff(S^1 \times I) \to \Diff(S^1)$ is a homotopy equivalence
by \cite[Th\'eor\`eme 1]{Gramain}, and the restriction map $\Map(S^1 \times I, X) \to LX$ is an
equivalence by retracting the cylinder down to one of its ends, it follows that the source map
$\Mor^0_{\CYL(X)} \to \Obj_{\CYL(X)}$ and the extension map $e$ are both homotopy equivalences.
Thus $\gamma$ is a level-wise homotopy equivalence of simplicial spaces, and hence it induces a weak
equivalence of geometric realizations.
\end{proof}

The $E_\infty$ product on $\SURF(X)$ restricts to an $E_\infty$ product on $\CYL(X)$, but this
subcategory is not group complete.  In light of the above lemma, we have:

\begin{proposition}\label{prop:cyl-group-completion}
There is a weak equivalence $\Omega B |\CYL(X)| \simeq Q(LX_{hS^1})_+$.
\end{proposition}

We are now ready to prove Theorem \ref{thm2}.
\begin{proof}
Since the space of objects of $\CIRC(X)$ is just $\mathbb{R}$, which is contractible, it follows
that the inclusion 
\[
\Mor_{\CIRC(X)} \hookrightarrow \Omega |\CYL(X)|
\]
is a homotopy equivalence.  We will show that the diagram
\begin{equation}\label{eq:final-square}
\begin{diagram}
\node{\Mor_{\CIRC(X)}} \arrow{s,l}{\mathrm{PT}} \arrow{e,t}{\simeq} \node{\Obj_{\CYL(X)}} \arrow{e} \node{|\SURF(X)|} \arrow{s,lr}{\simeq}{\mathrm{PT}}\\
\node{\Omega QX_+} \arrow[2]{e} \node[2]{\Omega^{\infty-1} MTSO(2)\wedge X_+}
\end{diagram}
\end{equation}
is homotopy commutative.  The desired result will follow from this diagram combined with Proposition
\ref{prop:cyl-group-completion} and Theorem \ref{big-theorem-X}. 

Consider $BSO(1) = S(\R^\infty) \simeq *$.  We replace $\Omega QX_+$ with the homotopy equivalent
space $\Omega^\infty MTSO(1) \wedge X_+$. Correspondingly, we replace the bottom horizontal arrow
 with the map
\[
\Omega^\infty MTSO(1) \wedge X_+ \to \Omega^{\infty-1} MTSO(2) \wedge X_+
\]
induced by the map of spectra $MTSO(1) \to \Sigma MTSO(2)$ coming from the stabilisation map
$BSO(1) \to BSO(2)$.

The map $\Obj_{\CYL(X)} = \Obj_{\SURF(X)} \to |\SURF(X)|$ corresponds to thinking of a circle in
$\R^\infty$ as the germ of a straight cylinder in $\R^\infty \times \R$, and hence classifying the
tangent bundle in $BSO(1)$ and then mapping to $BSO(2)$ is the same as extending the circle to the
germ of a cylinder and then classifying the tangent bundle in $BSO(2)$.  It now follows that the
modified version of the square \eqref{eq:final-square} commutes strictly, and hence the original
square commutes up to homotopy.
\end{proof}

\bibliographystyle{amsalpha} 
\bibliography{circle-transfer-bib}

\end{document}